\documentclass[11pt,reqno]{amsart}
\usepackage{amscd,amssymb,amsmath,amsthm}
\usepackage[arrow,matrix]{xy}
\usepackage{graphicx}
\usepackage{epstopdf}
\usepackage{color}
\usepackage{multicol}
\usepackage{tikz}
\usetikzlibrary {positioning}
\newtheorem{thm}{Theorem}
\newtheorem{defn}{Definition}
\newtheorem{lemma}{Lemma}
\newtheorem{pro}{Proposition}
\newtheorem{rk}{Remark}

\numberwithin{equation}{section} 

\begin{document}
\title [A DISCRETE-TIME DYNAMICAL SYSTEM OF MOSQUITO POPULATION]
{A DISCRETE-TIME DYNAMICAL SYSTEM OF MOSQUITO POPULATION}

\author{Z.S. Boxonov}

\address{Z.\ S.\ Boxonov\\ V.I.Romanovskiy Institute of Mathematics of Uzbek Academy of Sciences,
Tashkent, Uzbekistan.}
\email {z.b.x.k@mail.ru}

\begin{abstract}
In this paper, we study a discrete-time dynamical system generated by the evolution operator of a wild mosquito population with specific rates of birth and emergence from larvae to adults. The death rates of larvae and adults are assumed to be constant. We find fixed points and under some conditions, on parameters, we show the global attractiveness of a fixed point.
\end{abstract}


 \maketitle

\section{Introduction}
Many problems of natural sciences can be solved by using nonlinear dynamical systems (see \cite{D}, \cite{UR}, \cite{Rpd}). In this paper, we consider a wild mosquito population without the presence of
sterile mosquitoes. For the simplified stage-structured mosquito population, we group the three aquatic stages into the larvae class by $x$, and divide the mosquito population into the larvae class and the adults, denoted by $y$. We assume that the density dependence exists only in the larvae stage \cite{J.Li}.

We let the birth rate, that is, the oviposition rate of adults be $\beta(\cdot)$; the rate of emergence from larvae to adults be a function of the larvae with the form of $\alpha(1-k(x))$,
where $\alpha>0$ is the maximum emergence rate, $0\leq k(x)\leq 1$, with $k(0)=0, k'(x)>0$, and $\lim\limits_{x\rightarrow \infty}k(x)=1$, is the functional response due to the intraspecific competition \cite{J}. We further assume a functional response for $k(x)$, as in \cite{J}, in the form $$k(x)=\frac{x}{1+x}.$$ We let the death rate of larvae be a linear function, denoted by $d_{0}+d_{1}x$, and
the death rate of adults be constant, denoted by $\mu$. The interactive dynamics for the wild mosquitoes are governed by the following system:
\begin{equation}\label{system}
\left\{%
\begin{array}{ll}
    \frac{dx}{dt}=\beta y-\frac{\alpha x}{1+x}-(d_{0}+d_{1}x)x,\\[3mm]
    \frac{dy}{dt}=\frac{\alpha x}{1+x}-\mu y. \\
\end{array}%
\right.\end{equation}
In this paper (as in \cite{Rpd} - \cite{RB3}) we study the discrete time dynamical systems associated to the system (\ref{system}).
Define the operator $W:\mathbb{R}^{2}\rightarrow \mathbb{R}^{2}$ by
\begin{equation}\label{systema}
\left\{%
\begin{array}{ll}
    x'=\beta y-\frac{\alpha x}{1+x}-(d_{0}+d_{1}x)x+x,\\[3mm]
    y'=\frac{\alpha x}{1+x}-\mu y+y
\end{array}%
\right.\end{equation}
where $\alpha >0, \beta >0, \mu >0,\ d_{0}\geq0,\ d_{1}\geq0.$

The dynamics of the operator (\ref{systema}) were studied in detail under conditions $\beta=\mu,\ d_{0}=d_{1}=0$  in \cite{BR1} and under conditions $\beta\neq\mu,\ d_{0}=d_{1}=0$ in \cite{BR2}.

In this paper, we study a discrete dynamical system of the operator $W$ given by system (\ref{systema}), which differs from \cite{BR1} and \cite{BR2}.

\section{Results}

Let $\mathbb{R}_{+}^{2}=\{(x,y): x,y\in\mathbb{R}, x\geq0, y\geq0\}.$

Note that the operator $W$ well defined on $\mathbb{R}^2_{+}\setminus \{(x, y): x=-1\}$. But to define a dynamical system of continuous operator as population we assume $x\geq0$ and $y\geq0$. Therefore, we choose parameters of the operator $W$ to
guarantee that it maps $\mathbb{R}^2_{+}$ to itself.
\begin{lemma}(see \cite{BR1}) If
\begin{equation}\label{parametr}
\alpha>0,\ \beta>0,\ 0<\mu\leq1,\ d_{0}>0,\ \alpha+d_{0}\leq1,\ d_{1}=0
\end{equation}
then the operator (\ref{systema}) maps the set $\mathbb{R}_{+}^{2}$ to itself.
\end{lemma}
In this case the system (\ref{systema}) becomes
\begin{equation}\label{syst}
W_{0}:\left\{%
\begin{array}{ll}
    x'=\beta y-\left(\frac{\alpha }{1+x}+d_{0}-1\right)x,\\[3mm]
    y'=\frac{\alpha x}{1+x}+(1-\mu)y.
\end{array}%
\right.\end{equation}

\subsection{Fixed points.}\

A point $z\in\mathbb{R}_{+}^{2}$ is called a fixed point of $W_{0}$ if
$W_{0}(z)=z$.

For fixed point of $W_{0}$ the following holds.
\begin{pro}\label{fixed} The fixed points for (\ref{syst}) are as follows:
\begin{itemize}
  \item If $\beta\leq\mu\left(1+\frac{d_{0}}{\alpha}\right)$ then the operator (\ref{syst}) has a unique fixed point $z=(0,0).$
  \item If $\beta>\mu\left(1+\frac{d_{0}}{\alpha}\right)$ then mapping (\ref{syst}) has two
fixed points with $$z_{1}=(0,0), \ \ z_{2}=\left(\frac{\alpha(\beta-\mu)}{\mu d_{0}}-1,\frac{\alpha(\beta-\mu)-\mu d_{0}}{\mu(\beta-\mu)}\right).$$
\end{itemize}
\end{pro}
\begin{proof}
We have to solve the following system
\begin{equation}\label{Fsystema}
\left\{%
\begin{array}{ll}
    x=\beta y-\left(\frac{\alpha }{1+x}+d_{0}-1\right)x,\\[3mm]
    y=\frac{\alpha x}{1+x}+(1-\mu)y
\end{array}%
\right.\end{equation}
It is easy to see that $x_{1}=0,\ y_{1}=0$ and $x_{2}=\frac{\alpha(\beta-\mu)}{\mu d_{0}}-1,\ y_{2}=\frac{\alpha}{\mu}-\frac{d_{0}}{\beta-\mu}$ are solutions to (\ref{Fsystema}).
If $\beta\leq\mu\left(1+\frac{d_{0}}{\alpha}\right)$ then $x_{2}\notin\mathbb{R}_{+},$ otherwise $x_{2}\in\mathbb{R}_{+}.$
\end{proof}

\subsection{Types of the fixed points}\

Now we shall examine the type of the fixed points.
\begin{defn}\label{d1}
(see\cite{D}) A fixed point $s$ of the operator $W$ is called
hyperbolic if its Jacobian $J$ at $s$ has no eigenvalues on the
unit circle.
\end{defn}

\begin{defn}\label{d2}
(see\cite{D}) A hyperbolic fixed point $s$ called:

1) attracting if all the eigenvalues of the Jacobi matrix $J(s)$
are less than 1 in absolute value;

2) repelling if all the eigenvalues of the Jacobi matrix $J(s)$
are greater than 1 in absolute value;

3) a saddle otherwise.
\end{defn}
To find the type of a fixed point of the operator (\ref{syst})
we write the Jacobi matrix at $z=(x, y)$

$$J(z)=\left(%
\begin{array}{cc}
  1-d_{0}-\frac{\alpha}{(1+x)^2} & \beta \\
  \frac{\alpha}{(1+x)^2} & 1-\mu \\
\end{array}%
\right).$$

We calculate eigenvalues of Jacobian matrix at the fixed point ($0, 0)$. If we denote $1-\lambda=\Lambda$ then we obtain
\begin{equation}\label{eq00}
\Lambda^2-(d_{0}+\alpha+\mu)\Lambda+\mu(d_{0}+\alpha)-\alpha\beta=0
\end{equation}
By (\ref{eq00}) we have
\begin{equation}\label{lambda12}
\begin{split}
\Lambda_{1}=1-\lambda_1=\frac{1}{2}\left(\mu+d_{0}+\alpha+\sqrt{(\mu+d_{0}+\alpha)^2-4(\mu(d_{0}+\alpha)-\alpha\beta)}\right),\\
\Lambda_{2}=1-\lambda_2=\frac{1}{2}\left(\mu+d_{0}+\alpha-\sqrt{(\mu+d_{0}+\alpha)^2-4(\mu(d_{0}+\alpha)-\alpha\beta)}\right).
\end{split}
\end{equation}
The inequality $|\lambda_{1,2}|<1$ is equivalent to $0<\Lambda_{1,2}<2.$

Since $0<d_{0}+\alpha+\mu\leq2$ (see (\ref{parametr})), we have $0<\Lambda_{1}+\Lambda_{2}\leq2$ and $(\Lambda_{1}-2)+(\Lambda_{2}-2)=d_{0}+\alpha+\mu-4<0$. If $\beta<\mu\left(1+\frac{d_{0}}{\alpha}\right)$ then $\Lambda_{1}\cdot\Lambda_{2}=\alpha(\mu(1+\frac{d_{0}}{\alpha})-\beta)>0$ and $(\Lambda_{1}-2)\cdot(\Lambda_{2}-2)=\Lambda_{1}\cdot\Lambda_{2}-2(\Lambda_{1}+\Lambda_{2})+4=\alpha(\mu(1+\frac{d_{0}}{\alpha})-\beta)+2(2-(d_{0}+\alpha+\mu))>0.$ So $0<\Lambda_{1,2}<2.$
Therefore, if $\beta<\mu\left(1+\frac{d_{0}}{\alpha}\right)$ the fixed point $(0, 0)$ is attracting.

If the parameters satisfy the condition $\beta=\mu\left(1+\frac{d_{0}}{\alpha}\right)$ or $\beta=\frac{1}{\alpha}\left(2-\mu)(2-d_{0}-\alpha\right)$, the solutions of the equation (\ref{eq00}) are $0$ and $2$, respectively. Thus,  if $\beta=\mu\left(1+\frac{d_{0}}{\alpha}\right)$ or $\beta=\frac{1}{\alpha}\left(2-\mu)(2-d_{0}-\alpha\right),$ then the fixed point $(0, 0)$ is non hyperbolic.

The inequality $|\lambda_{1,2}|>1$ is equivalent to $\Lambda_{1}<0$ or $\Lambda_{2}>2$ or $\Lambda_{1}>2,\ \Lambda_{2}<0.$

For the values of parameters given in (\ref{parametr}) the inequalities $\Lambda_1<0$ or $\Lambda_2>2$ do not hold. Next we check the case $\Lambda_1>2,\ \Lambda_2<0:$

\begin{equation}\label{l1l2}\left\{%
\begin{array}{ll}
    \Lambda_{1}=\frac{1}{2}\left(\mu+d_{0}+\alpha+\sqrt{(\mu+d_{0}+\alpha)^2-4(\mu(d_{0}+\alpha)-\alpha\beta)}\right)>2,\\[2mm]
    \Lambda_{2}=\frac{1}{2}\left(\mu+d_{0}+\alpha-\sqrt{(\mu+d_{0}+\alpha)^2-4(\mu(d_{0}+\alpha)-\alpha\beta)}\right)<0.
\end{array}%
\right.
\end{equation}

If $\beta>\mu\left(1+\frac{d_{0}}{\alpha}\right)$ then $\Lambda_{1}\cdot\Lambda_{2}=\alpha(\mu(1+\frac{d_{0}}{\alpha})-\beta)<0.$ So $\Lambda_{1}>0$ and $\Lambda_{2}<0.$

From the inequality $\Lambda_{1}>2$ we obtain $\beta>\mu\left(1+\frac{d_{0}}{\alpha}\right)+\frac{1}{\alpha}(4-2(\mu+d_{0}+\alpha)).$
Therefore, if $\mu\left(1+\frac{d_{0}}{\alpha}\right)<\beta<\mu\left(1+\frac{d_{0}}{\alpha}\right)+\frac{1}{\alpha}(4-2(\mu+d_{0}+\alpha))$ the fixed point $(0, 0)$ is saddle and if $\beta>\mu\left(1+\frac{d_{0}}{\alpha}\right)+\frac{1}{\alpha}(4-2(\mu+d_{0}+\alpha))$ the fixed point $(0, 0)$ is repelling.

Denote
 $$x^*=\frac{\alpha(\beta-\mu)}{\mu d_{0}}-1,\ y^*=\frac{\alpha(\beta-\mu)-\mu d_{0}}{\mu(\beta-\mu)}.$$

We calculate eigenvalues of Jacobian matrix at the fixed point ($x^*, y^*)$. If $1-\bar{\lambda}=\bar{\Lambda}$ then we get
\begin{equation}\label{eqxy}
\bar{\Lambda}^2-\left(\mu+d_{0}+\frac{\alpha}{(1+x^*)^2}\right)\bar{\Lambda}+\mu\left(d_{0}+\frac{\alpha}{(1+x^*)^2}\right)-\frac{\alpha\beta}{(1+x^*)^2}=0
\end{equation}
By equation (\ref{eqxy}) and the condition $\beta>\mu(1+\frac{d_{0}}{\alpha})$ we have
\begin{equation}\label{1+2}
0<\bar{\Lambda}_{1}+\bar{\Lambda}_{2}=\mu+d_{0}+\frac{\alpha}{(1+x^*)^2}<2\ \  (see\ (\ref{parametr})),
\end{equation}
\begin{equation}\label{1*2}
\bar{\Lambda}_{1}\cdot\bar{\Lambda}_{2}=\frac{\alpha}{(1+x^*)^2}\left(\mu\left(1+\frac{d_{0}}{\alpha}(1+x^*)^2\right)-\beta\right)=\frac{\alpha(\beta-\mu)x^*}{(1+x^*)^2}>0.
\end{equation}
 From (\ref{1+2}) and (\ref{1*2}) we obtain $0<\bar{\Lambda}_{1,2}<2.$ Therefore, if $\beta>\mu\left(1+\frac{d_{0}}{\alpha}\right)$ the fixed point $(x^*, y^*)$ is attracting.

Thus we have proven the following proposition.
\begin{pro}\label{type} The type of the fixed points for (\ref{syst}) are as follows:
\begin{itemize}
  \item[i)] if $\beta\leq\mu\left(1+\frac{d_{0}}{\alpha}\right)$,  then the operator (\ref{syst}) has unique fixed point $(0,0)$, the point
  $$(0,0)=\left\{\begin{array}{lll}
  attracting, \ \ \mbox{if} \ \ \beta<\mu\left(1+\frac{d_{0}}{\alpha}\right) \\[2mm]
  non-hyperbolic, \ \ \mbox{if} \ \  \beta=\mu\left(1+\frac{d_{0}}{\alpha}\right).
  \end{array}\right.$$
\item[ii)]if \ $\beta>\mu\left(1+\frac{d_{0}}{\alpha}\right)$, then the operator has two fixed points $(0,0)$, $(x^*,y^*)$, and the point $(x^*,y^*)$ is attracting,  the point
  $$(0,0)=\left\{\begin{array}{lll}
  repelling, \ \ \mbox{if} \ \ \beta>\mu\left(1+\frac{d_{0}}{\alpha}\right)+\alpha^* \\[2mm]
  saddle, \ \ \mbox{if} \ \ \mu\left(1+\frac{d_{0}}{\alpha}\right)<\beta<\mu\left(1+\frac{d_{0}}{\alpha}\right)+\alpha^* \\[2mm]
  non-hyperbolic, \ \ \mbox{if} \ \  \beta=\mu\left(1+\frac{d_{0}}{\alpha}\right)+\alpha^*.
  \end{array}\right.$$
  \end{itemize}
where $\alpha^*=\frac{1}{\alpha}\left(4-2(\alpha+\mu+d_{0})\right).$
\end{pro}

\subsection{The dynamics of (\ref{syst}).}\

The following theorem  describes the trajectory of any initial point $(x^{(0)}, y^{(0)})$ in $\mathbb{R}^2_{+}$.

\begin{thm}\label{pr} For the operator $W_{0}$ given by (\ref{syst}) (i.e. under condition (\ref{parametr}))
and for any initial point $(x^{(0)}, y^{(0)})\in \mathbb R^2_+$ the following hold:
\begin{itemize}
\item[(i)] If $y^{(n)}>\frac{\alpha}{\mu}$ for any natural number $n$ then $$\lim\limits_{n\to \infty}x^{(n)}=+\infty, \ \lim\limits_{n\to \infty}y^{(n)}=\frac{\alpha}{\mu};$$
\item[(ii)] If there exists $n_{0}$ number such that  $y^{(n_{0})}\leq\frac{\alpha}{\mu}$ and $\beta<\mu\left(1+\frac{d_{0}}{\alpha}\right)$ then $$\lim\limits_{n\to \infty}x^{(n)}=0, \ \lim\limits_{n\to \infty}y^{(n)}=0,$$
\end{itemize}
where $(x^{(n)}, y^{(n)})=W_0^n(x^{(0)}, y^{(0)})$, with $W_{0}^n$ is $n$-th iteration of $W_{0}$.
\end{thm}
\begin{proof} We have
\begin{equation}\label{x^ny^n} x^{(n)}=\beta y^{(n-1)}-\frac{\alpha x^{(n-1)}}{1+x^{(n-1)}}-d_{0}x^{(n-1)}+x^{(n-1)},\ y^{(n)}=\frac{\alpha x^{(n-1)}}{1+x^{(n-1)}}-\mu y^{(n-1)}+y^{(n-1)}.\end{equation}
First, we prove the assertion $(i).$ Let all values of $y^{(n)}$ are greater than $\frac{\alpha}{\mu}$. Then
$$y^{(n+1)}-y^{(n)}=\frac{\alpha x^{(n)}}{1+x^{(n)}}-\mu y^{(n)}<\alpha-\mu y^{(n)}=\mu(\frac{\alpha}{\mu}-y^{(n)})<0.$$ So $y^{(n)}$ is a decreasing sequence. Since  $y^{(n)}$ is decreasing and bounded from below we have:
\begin{equation}\label{y[n]geq}
\lim_{n\to \infty}y^{(n)}\geq\frac{\alpha}{\mu}.
\end{equation}
We estimate $y^{(n)}$ by the following:
 $$y^{(n)}=\frac{\alpha x^{(n-1)}}{1+x^{(n-1)}}+(1-\mu)y^{(n-1)}<\alpha+(1-\mu)y^{(n-1)}<\alpha+(1-\mu)(\alpha+(1-\mu)y^{(n-2)})$$
 $$<\alpha+\alpha(1-\mu)+(1-\mu)^2(\alpha+(1-\mu)y^{(n-3)})< ...<\alpha+\alpha(1-\mu)+\alpha(1-\mu)^2+...+\alpha(1-\mu)^{n-1}$$ $$+(1-\mu)^{n}y^{(0)}=\frac{\alpha}{\mu}+(1-\mu)^{n}(y^{(0)}-\frac{\alpha}{\mu}).$$
Thus $y^{(n)}<\frac{\alpha}{\mu}+(1-\mu)^{n}(y^{(0)}-\frac{\alpha}{\mu})$. Consequently
\begin{equation}\label{y[n]leq}
\lim_{n\to \infty}y^{(n)}\leq\frac{\alpha}{\mu}.
\end{equation}
 By (\ref{y[n]geq}) and (\ref{y[n]leq}) we have \begin{equation}\label{y[n]=a}
 \lim\limits_{n\to \infty}y^{(n)}=\frac{\alpha}{\mu}.
\end{equation}
From (\ref{y[n]=a}) and (\ref{x^ny^n}) it follows $\lim\limits_{n\to \infty}x^{(n)}=+\infty.$

Let's prove the assertion $(ii)$.
If $y^{(n-1)}\leq\frac{\alpha}{\mu}$ then $y^{(n)}<\frac{\alpha}{\mu}$. Indeed,
  $$y^{(n)}=(1-\mu)y^{(n-1)}+\frac{\alpha x^{(n-1)}}{1+x^{(n-1)}}\leq(1-\mu)\frac{\alpha}{\mu}+\frac{\alpha x^{(n-1)}}{1+x^{(n-1)}}$$
  $$=\frac{\alpha}{\mu}-\alpha(1-\frac{x^{(n-1)}}{1+x^{(n-1)}})<\frac{\alpha}{\mu}.$$
There exists $n_0$ such that the sequence $y^{(n)}$  is  less than $\frac{\alpha}{\mu}$ for $n> n_0$.

The proof is based on the following lemmas.
\begin{lemma}\label{xnborder}
There exists $m_0$ such that the sequence $x^{(n)}$  is  less than $\frac{\alpha\beta}{\mu d_{0}}$ for $n> m_0$.
\end{lemma}
\begin{proof}
Assume that all values of $x^{(n)}$ are not less than $\frac{\alpha\beta}{\mu d_{0}}$. Then
$$x^{(n+1)}-x^{(n)}=\beta y^{(n)}-\frac{\alpha x^{(n)}}{1+x^{(n)}}-d_{0}x^{(n)}<\beta\cdot\frac{\alpha}{\mu}-d_{0} x^{(n)}=d_{0}\left(\frac{\alpha\beta}{\mu d_{0}}-x^{(n)}\right)\leq0.$$ So $x^{(n)}$ is a decreasing sequence. Since  $x^{(n)}$ is decreasing and bounded from below we have:
\begin{equation}\label{x[n]>a}
\lim_{n\to \infty}x^{(n)}\geq\frac{\alpha\beta}{\mu d_{0}}.
\end{equation}
We estimate $x^{(n)}$ by the following:
$$x^{(n)}=\beta y^{(n-1)}-\frac{\alpha x^{(n-1)}}{1+x^{(n-1)}}-d_{0}x^{(n-1)}+x^{(n-1)}<\frac{\alpha\beta}{\mu}+(1-d_{0})x^{(n-1)}<$$
$$<\frac{\alpha\beta}{\mu}+(1-d_{0})\left(\frac{\alpha\beta}{\mu}+(1-d_{0})x^{(n-2)}\right)<$$
$$<...<\frac{\alpha\beta}{\mu}+\frac{\alpha\beta}{\mu}(1-d_{0})+\frac{\alpha\beta}{\mu}(1-d_{0})^2+...+\frac{\alpha\beta}{\mu}(1-d_{0})^{n-1}+(1-d_{0})^{n}x^{(0)}=$$
$$=\frac{\alpha\beta}{\mu d_{0}}+(1-d_{0})^{n}\left(x^{(0)}-\frac{\alpha\beta}{\mu d_{0}}\right).$$
Thus $x^{(n)}<\frac{\alpha\beta}{\mu d_{0}}+(1-d_{0})^{n}\left(x^{(0)}-\frac{\alpha\beta}{\mu d_{0}}\right).$ From condition (\ref{parametr}) it follows $d_{0}\in(0, 1).$ Consequently
\begin{equation}\label{x[n]<a}
\lim_{n\to \infty}x^{(n)}\leq\frac{\alpha\beta}{\mu d_{0}}.
\end{equation}
 By (\ref{x[n]>a}) and (\ref{x[n]<a}) we have
 \begin{equation}\label{x[n]=a}
 \lim\limits_{n\to \infty}x^{(n)}=\frac{\alpha\beta}{\mu d_{0}}.
\end{equation}
Existing the limit of the sequence $x^{(n)}$ and  (\ref{x^ny^n}) derive the existence of the limit of the sequence $y^{(n)}$. By (\ref{x^ny^n}) we have
 \begin{equation}\label{x[n]=b}
 \lim\limits_{n\to \infty}x^{(n)}=\frac{\alpha(\beta-\mu)}{\mu d_{0}}-1.
\end{equation}
By (\ref{x[n]=a}) and (\ref{x[n]=b}) our assumption is false. Hence, for all $n>m_{0}$, there exists $m_0$ such that $x^{(n)}$  is less than $\frac{\alpha\beta}{\mu d_{0}}$.

If $x^{(n-1)}<\frac{\alpha\beta}{\mu d_{0}}$ then $x^{(n)}<\frac{\alpha\beta}{\mu d_{0}}$. Indeed,
 $$x^{(n)}=\beta y^{(n-1)}-\frac{\alpha x^{(n-1)}}{1+x^{(n-1)}}-d_{0}x^{(n-1)}+x^{(n-1)}<\frac{\alpha\beta}{\mu}+(1-d_{0})x^{(n-1)}<$$
 $$<\frac{\alpha\beta}{\mu}+(1-d_{0})\frac{\alpha\beta}{\mu d_{0}}=\frac{\alpha\beta}{\mu d_{0}}.$$
\end{proof}

\begin{lemma}\label{border} The sequences $x^{(n)}$ and $y^{(n)}$ are bounded.\end{lemma}
\begin{proof} It is obvious to see that the sequence $x^{(n)}$ is bounded due to Lemma \ref{xnborder}.

For the sequence $y^{(n)}$ the inequality $y^{(n)}\leq\frac{\alpha}{\mu}+(1-\mu)^{n}(y^{(0)}-\frac{\alpha}{\mu})$ holds.
Thus if the initial point $y^{(0)}>\frac{\alpha}{\mu}$ then for any $m\in\mathbb{N}$ we have $0\leq y^{(m)}\leq y^{(0)}$. Moreover, if $y^{(0)}<\frac{\alpha}{\mu}$ then for any $m\in\mathbb{N}$ we have $0\leq y^{(m)}\leq \frac{\alpha}{\mu}.$
\end{proof}
Let
 \begin{equation}\label{beta<mu}
 \beta<\mu(1+\frac{d_{0}}{\alpha}).
 \end{equation}
\begin{lemma}\label{sequence} For sequences $x^{(n)}$ and $y^{(n)}$ under condition (\ref{beta<mu})  the following statements hold:
\begin{itemize}
  \item[1)] For any $m\in\mathbb{N}$ the inequalities
   $x^{(m)}<x^{(m+1)}$ and $y^{(m)}<y^{(m+1)}$ can not be satisfied at the same time;
  \item[2)] If  $x^{(m-1)}>x^{(m)}$, $y^{(m-1)}>y^{(m)}$ for some $m\in\mathbb{N}$ then $x^{(m)}>x^{(m+1)}$, $y^{(m)}>y^{(m+1)};$
  \item[3)] For any $m\in\mathbb{N}$ the inequalities $x^{(m)}>x^{(m+1)}$, $y^{(m)}<y^{(m+1)}$ can not be satisfied at the same time;
  \item[4)] For any $m\in\mathbb{N}$ the inequalities $x^{(m)}<x^{(m+1)}$, $y^{(m)}>y^{(m+1)}$ can not be satisfied at the same time;
 \item[5)] For any $m\in\mathbb{N}$ the inequalities $x^{(m-1)}<x^{(m)}, x^{(m)}>x^{(m+1)}, y^{(m-1)}>y^{(m)}, y^{(m)}<y^{(m+1)}$  can not be satisfied at the same time;
 \item[6)] For any $m\in\mathbb{N}$ the inequalities $x^{(m-1)}>x^{(m)}, x^{(m)}<x^{(m+1)}, y^{(m-1)}<y^{(m)}, y^{(m)}>y^{(m+1)}$  can not be satisfied at the same time.
\end{itemize}
\end{lemma}
\begin{proof}
\begin{itemize}
  \item[1)] Assume for any $m\in\mathbb{N}$ and condition (\ref{beta<mu}) the inequalities
   $x^{(m)}<x^{(m+1)}$ and $y^{(m)}<y^{(m+1)}$ are satisfied at the same time. Then since $x^{(m)}$ and $y^{(m)}$ are increasing and bounded (see Lemma \ref{border}) there exist their limits $x^*\neq0$, $y^*\neq0$ respectively. By (\ref{syst}) and (\ref{beta<mu}) we obtain $x^*=0, y^*=0.$ This contradiction shows that if for any $m\in\mathbb{N}$ and condition (\ref{beta<mu}) the inequalities $x^{(m)}<x^{(m+1)}$ and $y^{(m)}<y^{(m+1)}$ can not be satisfied at the same time.
  \item[2)] Since the function $u(x)=x\left(1-d_{0}-\frac{\alpha}{1+x}\right)$ is monotonically increasing $(x\geq0)$ and by $x^{(m-1)}-x^{(m)}>0$ we have
  $$x^{(m-1)}\left(1-d_{0}-\frac{\alpha}{1+x^{(m-1)}}\right)-x^{(m)}\left(1-d_{0}-\frac{\alpha}{1+x^{(m)}}\right)>0$$ and
  $$\frac{x^{(m-1)}}{1+x^{(m-1)}}>\frac{x^{(m)}}{1+x^{(m)}}.$$
 Then
 $$x^{(m)}-x^{(m+1)}=\beta(y^{(m-1)}-y^{(m)})+x^{(m-1)}\left(1-d_{0}-\frac{\alpha}{1+x^{(m-1)}}\right)-$$$$-x^{(m)}\left(1-d_{0}-\frac{\alpha}{1+x^{(m)}}\right)>0,$$ $$y^{(m)}-y^{(m+1)}=\alpha\left(\frac{x^{(m-1)}}{1+x^{(m-1)}}-\frac{x^{(m)}}{1+x^{(m)}}\right)+(1-\mu)(y^{(m-1)}-y^{(m)})>0.$$
  \item[3)] Assume  under  condition (\ref{beta<mu}) the inequalities  $x^{(m)}>x^{(m+1)}$, $y^{(m)}<y^{(m+1)}$ are satisfied at the same time.
      Then since $x^{(n)}$ is decreasing and bounded; $y^{(n)}$ is increasing and bounded (see Lemma \ref{border}) there exist their limits $x^*$, $y^*\neq0$ respectively. By (\ref{syst}) we obtain
$$\lim_{m\rightarrow\infty}x^{(m)}=0, \ \ \lim_{m\rightarrow\infty}y^{(m)}=0.$$
This contradiction shows that for  condition (\ref{beta<mu}) the inequalities $x^{(m)}>x^{(m+1)}$, $y^{(m)}<y^{(m+1)}$ can not be satisfied at the same time.
    \item[4)] Assume under condition (\ref{beta<mu}) the inequalities $x^{(m)}<x^{(m+1)}$, $y^{(m)}>y^{(m+1)}$ are satisfied at the same time, i.e. $x^{(m)}$ is increasing and $y^{(m)}$ is decreasing. Thus from Lemma \ref{border} we conclude that $x^{(m)}$ and $y^{(m)}$ has a finite limit. By (\ref{syst}) we have
$$\lim_{m\rightarrow\infty}x^{(m)}=0, \ \ \lim_{m\rightarrow\infty}y^{(m)}=0.$$
But this is a contradiction to $\lim\limits_{m\rightarrow\infty}x^{(m)}\neq0$. This completes proof of part 4.
    \item[5)] Assume for any $m\in\mathbb{N}$ one has  $x^{(m-1)}<x^{(m)}, x^{(m)}>x^{(m+1)}, y^{(m-1)}>y^{(m)}, y^{(m)}<y^{(m+1)}$. Then $x^{(m-1)}<x^{(m+1)}, y^{(m-1)}>y^{(m+1)}$. Moreover, for each $k\in\mathbb{N}$ for  $m=2k+1$ we have  $x^{(2k-2)}<x^{(2k)}, y^{(2k-2)}>y^{(2k+2)}$, and for $m=2k$ we have $x^{(2k-1)}>x^{(2k+1)}, y^{(2k-1)}<y^{(2k+1)}$. But by Lemma \ref{sequence} parts 3) and 4) these inequalities do not hold for any $k\in\mathbb{N}$ (Example See Fig \ref{fig.1}).
    \item[6)] Similar to the proof of part 5.
\end{itemize}
\end{proof}

\begin{figure}[h!]
\includegraphics[width=0.7\textwidth]{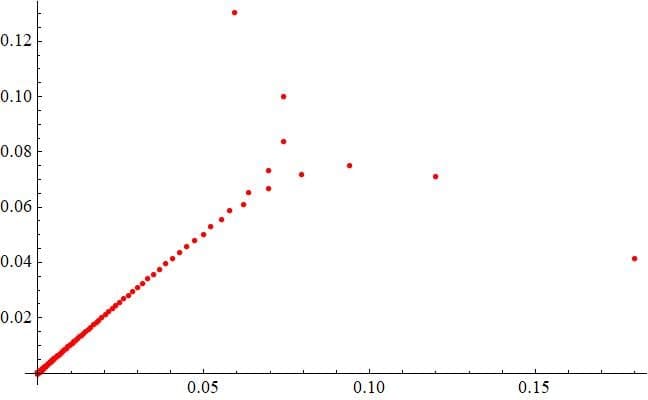}\\
\caption{$\alpha=0.8,\ \beta=0.9,\ \mu=0.8,\ d_{0}=0.2,\ x^{(0)}=0.002,\ y^{(0)}=0.2$}\label{fig.1}
\end{figure}

\begin{lemma}\label{monoton} If $\beta<\mu(1+\frac{d_{0}}{\alpha})$ then for all $n>n_{0}$, there exists  $n_0$ such that the sequences $x^{(n)}$ and $y^{(n)}$  are decreasing.
\end{lemma}
\begin{proof}  Easily deduced from Lemma \ref{sequence}.
\end{proof}
Now we continue to prove assertion (ii) of the theorem.

Thus for condition (\ref{beta<mu}) the sequences $x^{(n)}$ and $y^{(n)}$ has limits(see Lemma \ref{border} and Lemma \ref{monoton}).
Consequently, by (\ref{syst}) we get $\lim\limits_{n\rightarrow\infty}x^{(n)}=0$, $\lim\limits_{n\rightarrow\infty}y^{(n)}=0$.
\end{proof}
\begin{rk} For the case $\beta\geq\mu(1+\frac{d_{0}}{\alpha})$ the global behaviour of the fixed points is not studied.
\end{rk}

\section {A normalized (\ref{syst}) opertor}
In this section we consider the normalized version of the operator (\ref{syst}), i.e.,
\begin{equation}\label{systemnormal}
U:\left\{%
\begin{array}{ll}
    x'=\frac{(\beta y+(1-d_{0})x)(1+x)-\alpha x}{(1+x)((\beta-\mu+1)y+(1-d_{0})x)},\\[2mm]
    y'=\frac{\alpha x+(1+x)(1-\mu)y}{(1+x)((\beta-\mu+1)y+(1-d_{0})x)}.
\end{array}%
\right.\end{equation}

Denote $$S=\{(x,y)\in\mathbb{R}^2_{+}:x+y=1\}.$$
 From conditions (\ref{parametr}) one gets $x'\geq0$ and $y'\geq0$ moreover $x'+y'=1.$
Hence $U:S\rightarrow S.$

Using $x+y=1$, from (\ref{systemnormal}) we get
\begin{equation}\label{T}
T:x'=\frac{(1-d_{0}-\beta)x^2+(1-d_{0}-\alpha)x+\beta}{(\mu-\beta-d_{0})x^2+(1-d_{0})x+\beta-\mu+1}.
\end{equation}
Denote $$S^*=[0,1].$$

\begin{pro} If conditions (\ref{parametr}) are satisfied then the function $T$ (defined by (\ref{T}))  maps $S^*$ to itself.
\end{pro}
\begin{proof} We want to show that if $x\in S^*$ then $x'=T(x)\in S^*.$
Let $$h_{1}=(1-d_{0}-\beta)x^2+(1-d_{0}-\alpha)x+\beta,\ h_{2}=(\mu-\beta-d_{0})x^2+(1-d_{0})x+\beta-\mu+1$$
Then $h(x)=h_{2}-h_{1}=(\mu-1)x^2+\alpha x+1-\mu.$
Since $h(0)=1-\mu>0$ and $h(1)=\alpha>0$ for each $x\in[0,1]$ we have $h(x)\geq0,$ i.e., $x'\leq1.$ By
$$h_{1}=(1-d_{0}-\beta)x^2+(1-d_{0}-\alpha)x+\beta=(1-d_{0})x^2+\beta (1-x^2)+(1-d_{0}-\alpha)x\geq0,$$ $$h_{2}=(\mu-\beta-d_{0})x^2+(1-d_{0})x+\beta-\mu+1=(1+x)((\beta-\mu+1)(1-x)+(1-d_{0})x)>0$$
we have $x'\geq0,\ y'\geq0.$ Therefore under conditions (\ref{parametr}) we get $T:S^*\mapsto S^*$.
\end{proof}

\subsection{Fixed points.}\
For fixed point of $T$ the following lemma holds.
\begin{lemma}\label{lemmafix} (\ref{T}) has unique fixed point $x^*$.
\begin{equation}\label{efix}
x^*=\left\{%
\begin{array}{ll}
    \frac{\sqrt{\alpha^2+4\beta^2}-\alpha}{2\beta}, \ \ \ \ \ \ \ \ \ \ \ \ \ \ \ \ \ \ \ \ \ \ \ \ \ \ \ \ if \ \beta+d_{0}=\mu  \\[2mm]
    \frac{\sigma(\beta^2-3(\alpha-\mu+\beta+d_{0})(\mu-\beta-d_{0}))+1/\sigma-\beta}{3(\mu-\beta-d_{0})},\ \ \ if \ \beta+d_{0}\neq\mu.
\end{array}%
\right.\end{equation}
where $\sigma=\sqrt[3]{2/(9\beta e(\alpha+2e)-2\beta^3+\sqrt{4(-\beta^2+3e(\alpha-e))^3+(9\beta e(\alpha+2e)-2\beta^3)^2})},$
$e=\mu-\beta-d_{0}.$
\end{lemma}

\begin{proof}
Finding fixed points of function (\ref{T}) leads to the following equation.
\begin{equation}\label{fixT}
(\mu-\beta-d_{0})x^3+\beta x^2+(\beta-\mu+d_{0}+\alpha)x-\beta=0.
\end{equation}
If $\beta+d_{0}=\mu$ then $x^*_{1,2}=\frac{-\alpha\pm\sqrt{\alpha^2+4\beta^2}}{2\beta}.$ Since these roots should be in $S^*$, one can check that $x^*_{1}=\frac{-\alpha+\sqrt{\alpha^2+4\beta^2}}{2\beta}\in S^*,\ x^*_{2}=\frac{-\alpha-\sqrt{\alpha^2+4\beta^2}}{2\beta}\notin S^*.$

Let $\beta+d_{0}\neq\mu.$ Consider
$$l(x)=(\mu-\beta-d_{0})x^3+\beta x^2+(\beta-\mu+d_{0}+\alpha)x-\beta,$$
$$l'(x)=3(\mu-\beta-d_{0})x^2+2\beta x+\beta-\mu+d_{0}+\alpha,$$
$$l''(x)=6(\mu-\beta-d_{0})x+2\beta,$$
$$l'''(x)=6(\mu-\beta-d_{0}).$$
By $N(c)$ we denote the number of variations in sign in the ordered sequence of numberes
$$l(c), l'(c), l''(c), l'''(c).$$
If $\mu-\beta-d_{0}>0$ then $N(0)=1$ and $N(1)=0.$ If $\mu-\beta-d_{0}<0$ then $N(0)=2$ and $N(1)=1.$
Thus, according to the Budan-Foruier theorem, equation (\ref{fixT}) has a unique real root in the set $S^*$. It has the form (\ref{efix}).
\end{proof}

\subsection{Types of the fixed point.}\

Suppose $x^*$ is a fixed point for $T$. For one dimensional dynamical systems it is known that $x^*$ is an
attracting fixed point if $|T'(x^*)|<1$. The point $x^*$ is a repelling fixed point if $|T'(x^*)|>1$. Finally, if $|T'(x^*)|=1$, the fixed point is saddle \cite{D}.

Let us calculate the derivative of $T(x)$ at the fixed point $x^*$.

If $\beta+d_{0}=\mu$ then we have
$$T(x)=\frac{1}{1-d_{0}}\left((1-\mu)x+\beta-\alpha+\frac{\alpha}{x+1}\right),$$
$$T'(x)=\frac{1}{1-d_{0}}\left(1-\mu-\frac{\alpha}{(x+1)^2}\right).$$
$$1+x^{*}=1+\frac{-\alpha+\sqrt{\alpha^{2}+4\beta^{2}}}{2\beta}=\frac{2\alpha}{\alpha-2\beta+\sqrt{\alpha^{2}+4\beta^{2}}}.$$
Thus, we obtain
$T'(x^*)=1-\frac{\alpha^{2}+4\beta^{2}+(\alpha-2\beta)\sqrt{\alpha^{2}+4\beta^{2}}}{2\alpha(1-d_{0})}$.

If $|T'(x^{*})|<1$, then we have $$0<\alpha^2+4\beta^2+(\alpha-2\beta)\sqrt{\alpha^2+4\beta^2}<4\alpha(1-d_{0}).$$
If $\alpha\geq2\beta$ then $$0<\alpha^2+4\beta^2+(\alpha-2\beta)\sqrt{\alpha^2+4\beta^2}<\alpha^2+\alpha^2+\alpha\sqrt{\alpha^2+\alpha^2}<4\alpha^2<4\alpha(1-d_{0})$$ under the conditions (\ref{parametr}).
If $\alpha<\beta$ then from the conditions (\ref{parametr}) it follows
   $$0<\alpha^2+4\beta^2+(\alpha-2\beta)\sqrt{\alpha^2+4\beta^2}<(2\beta-\alpha)^2-(2\beta-\alpha)\sqrt{\alpha^2+4\beta^2}+4\alpha\beta=$$
   $$=(2\beta-\alpha)(2\beta-\alpha-\sqrt{\alpha^2+4\beta^2})+
   4\alpha\beta<4\alpha\beta<4\alpha(1-d_{0}).$$
Thus under condition (\ref{parametr}) the inequality $|T'(x^{*})|<1$ is always true.\\
Let $\beta+d_{0}\neq\mu$. Let's look at the derivative of the function $T(x)$. For a fixed point $x^*$ of the function $T(x)$ the following equality holds.
$$x^*=\frac{(1-d_{0}-\beta)x^{*2}+(1-d_{0}-\alpha)x^*+\beta}{(\mu-\beta-d_{0})x^{*2}+(1-d_{0})x^*+\beta-\mu+1} \Rightarrow$$  $$\Rightarrow (1-d_{0}-\beta)x^{*2}+(1-d_{0}-\alpha)x^*+\beta=x^*((\mu-\beta-d_{0})x^{*2}+(1-d_{0})x^*+\beta-\mu+1).$$
We have
\begin{equation}\label{typeT}
T'(x^*)=\frac{-2(\mu-\beta-d_{0})x{^*}^2+(1-d_{0}-2\beta)x^*+1-d_{0}-\alpha}{(\mu-\beta-d_{0})x{^*}^2+(1-d_{0})x^*+\beta-\mu+1}.
\end{equation}
First, we show that the expression (\ref{typeT}) is greater than -1.
From (\ref{typeT}) we get
$$T'(x^*)=-2+\frac{(3-3d_{0}-2\beta)x^*+3-d_{0}-\alpha+2\beta-2\mu}{(1+x^*)((\mu-d_{0}-\beta)x^*+\beta-\mu+1)}>-1+$$ $$+\frac{(3-3d_{0}-2\beta)x^*+3-d_{0}-\alpha+2\beta-2\mu-2((\mu-d_{0}-\beta)x^*+\beta-\mu+1)}
{2((\mu-d_{0}-\beta)x^*+\beta-\mu+1)}=$$ $$-1+\frac{(3-d_{0}-2\mu)x^*+1-d_{0}-\alpha}{2((\mu-d_{0}-\beta)x^*+\beta-\mu+1)}>-1+\frac{1-d_{0}-\alpha}{2((\mu-d_{0}-\beta)x^*+\beta-\mu+1)}>-1.$$

We rewrite (\ref{typeT}) using (\ref{fixT}):
\begin{equation}\label{typeT2}
T'(x^*)=\frac{(1-d_0)x{^*}^2+(1+2\beta-2\mu+d_{0}+\alpha)x^*-2\beta}{(1-d_{0}-\beta)x{^*}^2+(1-d_{0}-\alpha)x^*+\beta}.
\end{equation}
Next, let's show that  (\ref{typeT2}) is less than 1. Let $\mu-\beta-d_{0}>0$. Then
$$T'(x^*)-1=\frac{(1-d_0)x{^*}^2+(1+2\beta-2\mu+d_{0}+\alpha)x^*-2\beta}{(1-d_{0}-\beta)
x{^*}^2+(1-d_{0}-\alpha)x^*+\beta}-$$
$$-\frac{x^*((\mu-\beta-d_{0})x{^*}^2+(1-d_{0})x^*+\beta-\mu+1)}{(1-d_{0}-\beta)x{^*}^2+(1-d_{0}-\alpha)x^*+\beta}=$$
$$=\frac{(\beta-\mu+d_{0}+\alpha)x-2\beta-(\mu-\beta-d_{0})x{^*}^3}{(1-d_{0}-\beta)x{^*}^2+(1-d_{0}-\alpha)x^*+\beta}=$$
$$=\frac{-\beta x{^*}^2-(\mu-\beta-d_{0}) x{^*}^3-(\mu-\beta-d_{0})x{^*}^3-\beta}{(1-d_{0}-\beta)x{^*}^2+(1-d_{0}-\alpha)x^*+\beta}=$$$$=\frac{-\beta(1+x{^*}^2)-2(\mu-\beta-d_{0}) x{^*}^3}{(1-d_{0}-\beta)x{^*}^2+(1-d_{0}-\alpha)x^*+\beta}<0.$$

Let be $\mu-\beta-d_{0}<0$. From (\ref{fixT}) we obtain the inequality $$(\mu-\beta-d_{0})x{^*}^2+\beta x^*+\alpha-\mu+\beta+d_{0}>0.$$ From this inequality and when parameters satisfy conditions (\ref{parametr}) and $\mu-\beta-d_{0}<0$:
\begin{equation}\label{e<0}
3(\mu-\beta-d_{0})x{^*}^2+2\beta x^*+\alpha-\mu+\beta+d_{0}>0.
\end{equation}
From (\ref{typeT}) and (\ref{e<0}) we have the following:
$$T'(x^*)-1=\frac{-3(\mu-\beta-d_{0})x{^*}^2-2\beta x^*-\alpha+\mu-\beta-d_{0}}{(\mu-\beta-d_{0})x{^*}^2+(1-d_{0})x^*+\beta-\mu+1}<0.$$

Thus, if conditions (\ref{parametr}) are satisfied then the absolute value of  (\ref{typeT}) is always less than 1.

Thus for type of $x^{*}$ the following lemma holds.
\begin{lemma}\label{atr} The type of the fixed point $x^{*}$  for (\ref{T}) is attracting.
\end{lemma}

\subsection{Periodic  points.}\

A point $z$ in $U$ (see (\ref{systemnormal})) is called periodic point of $U$ if there exists $p$ so that $U^{p}(z)=z$. The smallest positive integer $p$ satisfying $U^{p}(z)=z$ is called the prime period or least period of the point $z.$

\begin{thm}\label{perT} For $p\geq 2$ the operator (\ref{systemnormal}) does not have any $p$-periodic point in the set $\mathbb{R}^{2}_{+}.$
\end{thm}
\begin{proof}
Let us first describe periodic points of $U$ with $p=2$. In this case the equation $U(U(z))=z$
can be reduced to description of 2-periodic points of the function $T$ defined in (\ref{T}), i.e.,to solution of the equation
\begin{equation}\label{per.2}
T(T(x))=x.\end{equation}
Note that the fixed points of $T$ are solutions to (\ref{per.2}), to find other solution we consider the equation
$$\frac{T(T(x))-x}{T(x)-x}=0,$$
simple calculations show that the last equation is equivalent to the following
\begin{equation}\label{PER}
Ax^2+Bx+C=0.\end{equation}
where $$A=(1-\beta)(\beta-2)+(\beta-\mu+1)(\beta-\mu)+d_{0}(5-2\beta-\mu-2d_{0}),$$ $$B=2(1-d_{0})(\alpha+d_{0}+\mu-2)-\alpha\beta,$$ $$C=(\beta-\mu+1)(\alpha+d_{0}+\mu-2)-\beta(2-\mu-d_{0}).$$
By (\ref{parametr}) we have $A+B+C=2(1-d_{0})(\alpha+\mu+2d_{0}-3)+(1-\mu)(\alpha+2d_{0}-2)<0$, $B<0$, $C<0$. Therefore since $x\geq 0$ the LHS of (\ref{PER}) is $<0$.
Consequently, the equation (\ref{PER}) does not have solution in $S^*$.
Thus function (\ref{T}) does not have any 2-periodic point in $S^*$.
Since $T$ is continuous on $S^*$ by Sharkovskii's theorem (\cite{D}, \cite{UR}) we have that
$T^p(x)=x$ does not have solution (except fixed) for all $p\geq 2$.
Hence it follows that the operator (\ref{systemnormal}) has no periodic points (except fixed) in the set $\mathbb{R}^2_{+}.$
\end{proof}

\subsection{The $\omega$-limit set}\

The problem of describing the $\omega$-limit set of a trajectory is of great importance in the theory of dynamical systems.

The following  describes the trajectory of any point $x_{0}$ in $S^*$.
\begin{thm}\label{wset} For the operator $T$ given by (\ref{T}), under condition (\ref{parametr}), for any $x_{0}\in S^*$ the following holds
 $$\lim_{n\to \infty}T^n(x_0)=\left\{\begin{array}{ll}
 \widetilde{x}, \ \ \mbox{if} \ \ \mu=\beta+d_{0}\\[2mm]
 \widehat{x}, \ \ \mbox{if} \ \ \ \mu\neq\beta+d_{0}
 \end{array}\right.$$
 where $T^n$ is $n$-th iteration of $T$, $\widetilde{x}=\frac{\sqrt{\alpha^2+4\beta^2}-\alpha}{2\beta}$, $\widehat{x}=\frac{\sigma(\beta^2-3(\alpha-\mu+\beta+d_{0})(\mu-\beta-d_{0}))+1/\sigma-\beta}{3(\mu-\beta-d_{0})}$ are fixed points.
\end{thm}
\begin{proof} Let's input the following denotations in the function (\ref{T}):
$$a=1-d_{0}-\beta, b=1-d_{0}-\alpha, e=\mu-\beta-d_{0}, f=\beta-\mu+1.$$

Let $e=0$. Then we write the function (\ref{T}) in the form:
$$T(x)=\frac{1}{1-d_{0}}\left((1-\mu)x+\beta-\alpha+\frac{\alpha}{x+1}\right).$$
Let us find minimum points of $T(x)$. By solving $T'(x)=0$ we have
$x_{min}=\sqrt{\frac{\alpha}{1-\mu}}-1$.
Obviously, the function $T$ is increasing in  $S^*=[0, 1]$ if parameters are $\alpha\leq1-\mu$ and decreasing if $\alpha\geq4(1-\mu)$, and  if $1-\mu<\alpha<4(1-\mu)$ then it is decreasing in $[0, x_{min}]$, increasing in  $[x_{min}, 1]$.

Let's write the function $e=f.$ (\ref{T}) as following
$$T(x)=\frac{1}{e}\left(a-\frac{2(1-\mu)+\alpha}{1+x}+\frac{\alpha}{(1+x)^2}\right).$$
By solving $T'(x)=0$ we have $x_{min}=\frac{\alpha-2(1-\mu)}{\alpha+2(1-\mu)}$.
In this case, it is clear that the function $T$ is increasing in $S^*$ if $\alpha\leq2(1-\mu),$ moreover, it is decreasing in $[0, x_{min}]$ and increasing in $[x_{min}, 1]$ when $\alpha>2(1-\mu).$
Let $e\neq f.$ Then we write th function (\ref{T}) in the form:
$$T(x)=\frac{a}{e}+\frac{\alpha}{(f-e)(1+x)}-\frac{d}{e(f-e)(f+ex)},$$
where $d=\alpha e^2+(f-e)((1-d_{0})(1-\mu)+\alpha e).$
If $d>0, f>e, (1-d_{0})(1-\mu)/f<\alpha<4(1-d_{0})(1-\mu)/(f-e)$ or $d>0, e>f, \alpha>(1-d_{0})(1-\mu)/f$ then the function $T(x)$ reaches a minimum in $S^*$, i.e., $x_{min}=\frac{\sqrt{d}-f\sqrt{\alpha}}{e\sqrt{\alpha}-\sqrt{d}}$, otherwise, the function $T(x)$ is monotone.

\begin{itemize}
\item[1)] \textbf{Case:} $x_{min}\not\in(0, 1).$ (cf. with proof of Lemma 3.4 of \cite{RAU}) $T(x)$ is an increasing function (see Fig.2). Here we consider the case when the function $T$ has unique fixed point $x^{*}.$ We have that the point $x^{*}$ is attractive, i.e., $|T'(x^*)|<1.$ Now we shall take arbitrary $x_{0}\in S^*$ and prove that $x_{n}=T(x_{n-1}),\ n\geq1$ converges as $n\rightarrow\infty.$ Consider the following partition $S^*=[0, 1]=[0, x^{*})\cup\{x^{*}\}\cup(x^{*}, 1].$ For any $x\in[0, x^{*})$ we have $x<T(x)<x^{*},$ since $T$ is an increasing function, from the last inequalities we get $x<T(x)<T^{2}(x)<T(x^{*})=x^{*}$ iterating this argument we obtain $T^{n-1}(1)<T^{n}(x)<x^{*},$ which for any $x_{0}\in[0;x^{*})$ gives $x_{n-1}<x_{n}<x^{*},$ i.e.,$x_{n}$ converges and its limit is a fixed point of $T,$ since $T$ has unique fixed point $x^{*}$ in $[0, x^{*}]$ we conclude that the limit is $x^{*}.$ For $x\in(x^{*};1]$ we have $1>x>T(x)>x^{*},$ consequently $x_{n}>x_{n+1},$ i.e.,\ $x_{n}$ converges and its limit is again $x^{*}.$

$T$ is a decreasing function (see Fig. 3). Let $g(x)=T(T(x)).$ $g$ is increasing since $g'(x)=T'(T(x))T'(x)>0.$ By Lemma \ref{lemmafix} and Theorem \ref{perT} we have that $g$ has at most unique fixed point (including $x^{*}$). Hence one can repeat the same argument of the proof of part $1)$ for the increasing function $g$ and complete the proof.

\textbf{Case:} $x_{min}\in(0, 1).$

a) Let $x_{min}<x^{*}.$ Consider the following partition $[0, 1]=[0, x_{min})\cup[x_{min}, 1].$
  The function $T(x)$ is decreasing in $[0, x_{min})$ and is increasing in $[x_{min}, 1]$ (see Fig. 4). For all $x\in[0, x_{min})$, $x<T(x),\ T(x)>x_{min}.$  For $T(x)\in [x_{min}, 1]$ it can be proved that $x_{n}$ converges to the attractive fixed point $x^{*}$ (see Lemma \ref{atr}) like previous case.

b) Let $x_{min}>x^{*}.$ Consider the following partition $[0, 1]=[0, x_{min})\cup[x_{min}, 1].$
For all $x\in[x_{min}, 1]$, $x>T(x)>T^{2}(x)>...>T^{k}(x),\ T^{k}(x)<x_{min}$ (see Fig. 5).
 If $T(x)\in [0, x_{min})$ then the sequence $x_{n}$ converges to $x^{*}$.
\end{itemize}
This completes the proof.
\end{proof}

\begin{center}
   \begin {tikzpicture} [scale=0.55]
   \draw[->, thick] (0,0) -- (0,10);
   \draw[->, thick] (0,0) -- (11,0);

   \draw[] (1,-0.1) -- (1,0.1);
   \draw[] (2,-0.1) -- (2,0.1);
   \draw[] (3,-0.1) -- (3,0.1);
   \draw[] (4,-0.1) -- (4,0.1);
   \draw[] (5,-0.1) -- (5,0.1);
   \draw[] (6,-0.1) -- (6,0.1);
   \draw[] (7,-0.1) -- (7,0.1);
   \draw[] (8,-0.1) -- (8,0.1);
   \draw[] (9,-0.1) -- (9,0.1);
   \draw[] (10,-0.1) -- (10,0.1);

   \draw[] (-0.1,1) -- (0.1,1);
   \draw[] (-0.1,2) -- (0.1,2);
   \draw[] (-0.1,3) -- (0.1,3);
   \draw[] (-0.1,4) -- (0.1,4);
   \draw[] (-0.1,5) -- (0.1,5);
   \draw[] (-0.1,6) -- (0.1,6);
   \draw[] (-0.1,7) -- (0.1,7);
   \draw[] (-0.1,8) -- (0.1,8);
   \draw[] (-0.1,9) -- (0.1,9);

  \draw[thick, green] (0,3.3).. controls (3,8) and (7,9) ..   (10, 9.4);
   \draw[thick, red] (0,0) --  (10, 10);
    \draw[->, thick] (1,0) -- (1,4.6);
    \draw[->, thick] (1,4.6) -- (4.6,4.6);
    \draw[->, thick] (4.6,4.6) -- (4.6,7.7);
    \draw[->, thick] (4.6,7.7) -- (7.7,7.7);
    \draw[->, thick] (7.7,7.7) -- (7.7,8.9);
    \draw[->, thick] (7.7,8.9) -- (8.9,8.9);
    \draw[->, thick] (10,10) -- (10,9.4);
    \draw[->, thick] (10,9.4) -- (9.4,9.4);
   \node[below] at (1,0){0.1};
   \node[below] at (2,0){0.2};
   \node[below] at (3,0){0.3};
   \node[below] at (4,0){0.4};
   \node[below] at (5,0){0.5};
   \node[below] at (6,0){0.6};
   \node[below] at (7,0){0.7};
   \node[below] at (8,0){0.8};
   \node[below] at (9,0){0.9};
   \node[below] at (10,0){1};
   \node[above] at (11,0){$x$};
   \node[left] at (0,1){0.1};
   \node[left] at (0,2){0.2};
   \node[left] at (0,3){0.3};
   \node[left] at (0,4){0.4};
   \node[left] at (0,5){0.5};
   \node[left] at (0,6){0.6};
   \node[left] at (0,7){0.7};
   \node[left] at (0,8){0.8};
   \node[left] at (0,9){0.9};
   \node[right] at (0,10){$y$};
\end{tikzpicture}
\begin {tikzpicture} [scale=0.55]
   \draw[->, thick] (0,0) -- (0,10);
   \draw[->, thick] (0,0) -- (11,0);
   \draw[] (1,-0.1) -- (1,0.1);
   \draw[] (2,-0.1) -- (2,0.1);
   \draw[] (3,-0.1) -- (3,0.1);
   \draw[] (4,-0.1) -- (4,0.1);
   \draw[] (5,-0.1) -- (5,0.1);
   \draw[] (6,-0.1) -- (6,0.1);
   \draw[] (7,-0.1) -- (7,0.1);
   \draw[] (8,-0.1) -- (8,0.1);
   \draw[] (9,-0.1) -- (9,0.1);
   \draw[] (10,-0.1) -- (10,0.1);
   \draw[] (-0.1,1) -- (0.1,1);
   \draw[] (-0.1,2) -- (0.1,2);
   \draw[] (-0.1,3) -- (0.1,3);
   \draw[] (-0.1,4) -- (0.1,4);
   \draw[] (-0.1,5) -- (0.1,5);
   \draw[] (-0.1,6) -- (0.1,6);
   \draw[] (-0.1,7) -- (0.1,7);
   \draw[] (-0.1,8) -- (0.1,8);
   \draw[] (-0.1,9) -- (0.1,9);

  \draw[thick, green] (0,9.3).. controls (3,5) and (7,4.6) ..   (10, 4.4);
   \draw[thick, red] (0,0) --  (10, 10);
    \draw[->, thick] (1,0) -- (1,8.1);
    \draw[->, thick] (1,8.1) -- (8.1,8.1);
    \draw[->, thick] (8.1,8.1) -- (8.1,4.6);
    \draw[->, thick] (8.1,4.6) -- (4.6,4.6);
    \draw[->, thick] (4.6,4.6) -- (4.6,5.5);
    \draw[->, thick] (4.6,5.5) -- (5.5,5.5);
    \draw[->, thick] (5.5,5.5) -- (5.5,5.1);
    \draw[->, thick] (5.5,5.1) -- (5.1,5.1);
   \node[below] at (1,0){0.1};
   \node[below] at (2,0){0.2};
   \node[below] at (3,0){0.3};
   \node[below] at (4,0){0.4};
   \node[below] at (5,0){0.5};
   \node[below] at (6,0){0.6};
   \node[below] at (7,0){0.7};
   \node[below] at (8,0){0.8};
   \node[below] at (9,0){0.9};
   \node[below] at (10,0){1};
   \node[above] at (11,0){$x$};
   \node[left] at (0,1){0.1};
   \node[left] at (0,2){0.2};
   \node[left] at (0,3){0.3};
   \node[left] at (0,4){0.4};
   \node[left] at (0,5){0.5};
   \node[left] at (0,6){0.6};
   \node[left] at (0,7){0.7};
   \node[left] at (0,8){0.8};
   \node[left] at (0,9){0.9};
   \node[right] at (0,10){$y$};
\end{tikzpicture}

{Figure 2.}\label{fig.2} \ \ \ \ \ \ \ \ \ \ \ \ \ \ \ \ \ \ \ \ \ \ \ \ \ \ {Figure 3.}
\end{center}

\begin{center}
   \begin {tikzpicture} [scale=0.55]
   \draw[->, thick] (0,0) -- (0,10);
   \draw[->, thick] (0,0) -- (11,0);
   \draw[] (1,-0.1) -- (1,0.1);
   \draw[] (2,-0.1) -- (2,0.1);
   \draw[] (3,-0.1) -- (3,0.1);
   \draw[] (4,-0.1) -- (4,0.1);
   \draw[] (5,-0.1) -- (5,0.1);
   \draw[] (6,-0.1) -- (6,0.1);
   \draw[] (7,-0.1) -- (7,0.1);
   \draw[] (8,-0.1) -- (8,0.1);
   \draw[] (9,-0.1) -- (9,0.1);
   \draw[] (10,-0.1) -- (10,0.1);
   \draw[] (-0.1,1) -- (0.1,1);
   \draw[] (-0.1,2) -- (0.1,2);
   \draw[] (-0.1,3) -- (0.1,3);
   \draw[] (-0.1,4) -- (0.1,4);
   \draw[] (-0.1,5) -- (0.1,5);
   \draw[] (-0.1,6) -- (0.1,6);
   \draw[] (-0.1,7) -- (0.1,7);
   \draw[] (-0.1,8) -- (0.1,8);
   \draw[] (-0.1,9) -- (0.1,9);

  \draw[thick, green] (0,6).. controls (2,2) and (7,4.6) ..   (10, 7);
   \draw[thick, red] (0,0) --  (10, 10);
    \draw[->, thick] (0.5,0) -- (0.5,5.2);
    \draw[->, thick] (0.5,5.2) -- (5.2,5.2);
    \draw[->, thick] (5.2,5.2) -- (5.2,4.3);
    \draw[->, thick] (5.2,4.3) -- (4.3,4.3);

   \node[below] at (1,0){0.1};
   \node[below] at (2,0){0.2};
   \node[below] at (3,0){0.3};
   \node[below] at (4,0){0.4};
   \node[below] at (5,0){0.5};
   \node[below] at (6,0){0.6};
   \node[below] at (7,0){0.7};
   \node[below] at (8,0){0.8};
   \node[below] at (9,0){0.9};
   \node[below] at (10,0){1};
   \node[above] at (11,0){$x$};
   \node[left] at (0,1){0.1};
   \node[left] at (0,2){0.2};
   \node[left] at (0,3){0.3};
   \node[left] at (0,4){0.4};
   \node[left] at (0,5){0.5};
   \node[left] at (0,6){0.6};
   \node[left] at (0,7){0.7};
   \node[left] at (0,8){0.8};
   \node[left] at (0,9){0.9};
   \node[right] at (0,10){$y$};
\end{tikzpicture}
 \begin {tikzpicture} [scale=0.55]
   \draw[->, thick] (0,0) -- (0,10);
   \draw[->, thick] (0,0) -- (11,0);
   \draw[] (1,-0.1) -- (1,0.1);
   \draw[] (2,-0.1) -- (2,0.1);
   \draw[] (3,-0.1) -- (3,0.1);
   \draw[] (4,-0.1) -- (4,0.1);
   \draw[] (5,-0.1) -- (5,0.1);
   \draw[] (6,-0.1) -- (6,0.1);
   \draw[] (7,-0.1) -- (7,0.1);
   \draw[] (8,-0.1) -- (8,0.1);
   \draw[] (9,-0.1) -- (9,0.1);
   \draw[] (10,-0.1) -- (10,0.1);
   \draw[] (-0.1,1) -- (0.1,1);
   \draw[] (-0.1,2) -- (0.1,2);
   \draw[] (-0.1,3) -- (0.1,3);
   \draw[] (-0.1,4) -- (0.1,4);
   \draw[] (-0.1,5) -- (0.1,5);
   \draw[] (-0.1,6) -- (0.1,6);
   \draw[] (-0.1,7) -- (0.1,7);
   \draw[] (-0.1,8) -- (0.1,8);
   \draw[] (-0.1,9) -- (0.1,9);

  \draw[thick, green] (0,9).. controls (2,4.2) and (7,2) ..   (10, 4);
   \draw[thick, red] (0,0) --  (10, 10);
    \draw[->, thick] (9.5,0) -- (9.5,3.7);
    \draw[->, thick] (9.5,3.7) -- (3.7,3.7);
    \draw[->, thick] (3.7,3.7) -- (3.7,4.5);
    \draw[->, thick] (3.7,4.5) -- (4.5,4.5);
    \draw[->, thick] (4.5,4.5) -- (4.5,4);
     \draw[->, thick] (4.5,4) -- (4,4);
          \draw[->, thick] (4,4) -- (4,4.3);

   \node[below] at (1,0){0.1};
   \node[below] at (2,0){0.2};
   \node[below] at (3,0){0.3};
   \node[below] at (4,0){0.4};
   \node[below] at (5,0){0.5};
   \node[below] at (6,0){0.6};
   \node[below] at (7,0){0.7};
   \node[below] at (8,0){0.8};
   \node[below] at (9,0){0.9};
   \node[below] at (10,0){1};
   \node[above] at (11,0){$x$};
   \node[left] at (0,1){0.1};
   \node[left] at (0,2){0.2};
   \node[left] at (0,3){0.3};
   \node[left] at (0,4){0.4};
   \node[left] at (0,5){0.5};
   \node[left] at (0,6){0.6};
   \node[left] at (0,7){0.7};
   \node[left] at (0,8){0.8};
   \node[left] at (0,9){0.9};
   \node[right] at (0,10){$y$};
\end{tikzpicture}

{Figure 4.} \ \ \ \ \ \ \ \ \ \ \ \ \ \ \ \ \ \ \ \ \ \ \ \ \ \ {Figure 5.}
\end{center}

\subsection{Discussion.}\

In this paper, we studied a discrete-time dynamical system generated by the
evolution operator of a wild mosquito population with specific rates of birth and emergence
from larvae to adults. The death rates of larvae and adults were assumed to be constant. We
found fixed points and under some conditions, on parameters, we have shown the global attractiveness
of a fixed point. Moreover, we have studied the dynamics of the evolution operator by normalizing the operator. For the normalized operator,  there was defined the type of the fixed point, it was proved that there is no any periodic points, consequently, we studied the set of limit points.

Let $(x, y)\in S$ be an initial state (the probability distribution on the set \{larvae, adults\}). Theorem \ref{wset} says that the population tends to the equilibrium state $(x^*, y^*)$ with the passage of time.
\subsection{Acknowledgement.}  I am grateful to Professor U.A.Rozikov for the comments and suggestions.

\end{document}